\documentclass[12pt,a4paper]{amsart}
\usepackage{amsmath}
\usepackage{amssymb,color}
\usepackage{amsfonts}
\usepackage{amsmath}
\usepackage{euscript}
\usepackage{enumerate}
\usepackage{graphics}
\usepackage[margin=2.9cm]{geometry}
\usepackage{tikz-cd}
\usetikzlibrary{cd}
 
\newtheorem{theorem}{Theorem}[section]

\newtheorem{proposition}[theorem]{Proposition}
\newtheorem{corollary}[theorem]{Corollary}

\newtheorem*{Theorem1'}{Theorem 1'}

\theoremstyle{definition}
\newtheorem{definition}[theorem]{Definition}

\theoremstyle{remark}

\numberwithin{equation}{section}

\newcommand \p{\mathrm{pd}}
\newcommand \ap{\mathrm{apd}}
\newcommand \di{\mathrm{dim}}

\newcommand \dm{{\mathrm {dim}}}

\newcommand \Z{{\mathbb Z}}

\newcommand \N{{\mathcal N}}

\newcommand{\z}{\zeta}
\newcommand{\E}{\widetilde{\text{Ext}}}
\newcommand{\e}{\text{Ext}}

\title{A NEW HOMOLOGICAL INVARIANT FOR MODULES}
\author{Mohammadali Izadi}
\begin{document}

\title{A NEW HOMOLOGICAL INVARIANT FOR MODULES}


\address{School of Science and Environment (Mathematics) \\ Grenfell Campus, Memorial University of Newfoundland \\
Corner Brook, NL, A2H 6P9, Canada } 

\email{mizadi@grenfell.mun.ca}

\subjclass[2010]{13D05, 13D03, 55N35.}
 
 \keywords{Gorenstein rings, Homological dimensions, Vogel cohomology.}

\maketitle
\begin{abstract}
 Let $R$ be a commutative Noetherian local ring with residue field $k$.
Using the structure of Vogel cohomology, for any finitely generated module $M$,
we introduce a new dimension, called $\zeta$-dimension, denoted by $\zeta\text{-}\dm_R M$. This
dimension is finer than Gorenstein dimension and has nice properties enjoyed by
homological dimensions. In particular, it characterizes Gorenstein rings in the
sense that: a ring $R$ is Gorenstein if and only if every finitely generated $R$-module
has finite $\zeta$-dimension. Our definition of $\zeta$-dimension offer a new homological
perspective on the projective dimension, complete intersection dimension of
Avramov et al. and G-dimension of Auslander and Bridger.
\end{abstract}

\section{Introduction}

Let $M$ be a finitely generated module over a commutative Noetherian ring $R$.
There are several homological invariants assigned to $M$. The most important one
is projective dimension $\text{pd}_R M$. Auslander and Bridger [AB] singled out the class
of R-modules of finite Gorenstein dimension ($G$-dimension) as a generalization of
modules of finite projective dimension. Avramov et al. [AGP] introduced the concept
of complete intersection (CI)-dimension. The purpose of this paper is to offer a
new dimension that we believe gives a new homological perspective on the aforementioned
concepts. We show that the above dimensions can be considered as special
cases of a much more general dimension that we shall call it $\zeta$-dimension. To introduce
this dimension, we use the structure of Vogel cohomology developed by Pierre
Vogel in 1980. The cohomology theory that he developed, associates to each pair
$(M,N)$ of modules, a sequence of $R$-modules $\E_R^n(M,N)$ for $n \in \Z$, and comes equipped with a natural transformation $\z^*(M,N):\e_R^n(M,N)\rightarrow \E_R^n(M,N)$ of
cohomology functors. For any $R$-module $M$ and any integer $i$, we let $\z^i(M)$ denote
the natural map $\e_R^i(M,k)\rightarrow \E_R^i(M,k)$, where $k$ is the residue field of $R$, and define $\z\text{-}\text{dim}_RM$ to be the infimum $n \in \N$, such that $\z^n(M)$ is epimorphism and $\z^n(M)$ is isomorphism for all $i > n$. Note that $\text{pd}_RM$, when it is finite, is equal
to the supremum of $i$'s such that $\e_R^i(M,k)\neq 0$. On the other hand, $\E_R^i(M,~)$ will vanish, for all $i \in \Z$. So in fact $\p_RM$ is equal to the infimum of $i$'s such that $\e_R^i(M,k)\cong \E_R^i(M,K)$, that is $\p_RM = \z\text{-}\text{dim}_RM$. Moreover, if $G\text{-}\di_RM$ is
finite, we shall show that it can be interpreted as the vanishing of certain $\z^i$’s (see
Theorem 1.4).

Moreover we examine the ability of $\z$-dimension to detect Gorensteinness of the
underlying ring: it is finite for all modules over a Gorenstein ring, conversely if
$\z\text{-}\di_R k <\infty$ the ring is Gorenstein. $\z$-dimension shares many basic properties
with other homological dimensions. In particular, it localizes. In an attempt to
find a lower bound for $\z$-dimension, in Theorem 1.8, for any finitely generated Rmodule
$M$ we obtain the following inequality $\text{apd}_RM \leq\z\text{-}\di_RM$, with equality if
$\z-\di_RM$ is finite. We recall that $\ap_RM$ is defined by the formula

$$\ap_RM=\text{sup} \Bigg\{ i\in \mathbb{N}_0 \begin{array}{c|c}  & \e_R^i(M,T)\neq 0~\text{for some finitely generated}\\ &~R\text{-}\text{module}~T~\text{with}~\p_RT<\infty\end{array}\Bigg\}.$$
So the place of $\z-di_R(M)$ in the hierarchy of homological dimensions is determined
as
$$\ap_RM\leq \z\text{-}\di_RM\leq \text{G}\text{-}\di_R(M)\leq \text{CI}\text{-}\di_R(M)\leq \p_R(M),$$
with equality to the left of any finite ones. No example of a module $M$ with
$\z\text{-}\dm_RM < \text{G-}\dm_RM$ is known at present to the authors.

Towards the end of the paper, we deal with the resolving property of the category
of modules of finite $\z$-dimension. Throughout the paper $(R,\mathfrak{m}, k)$ is a commutative Noetherian local ring with residue field $k$.

\section{$\z$-DIMENSION}

We begin by recalling the construction of Vogel cohomology. First let us mention
that, abusing notation we shall use the symbol $(A,B)$ for the graded Hom-functor applied to graded $R$-modules $A$ and $B$. Thus $(A,B)_n=\Pi_{i\in\Z} \text{Hom}_R(A_i,B_{i+n})$. The differential is defined on $(A,B)$ by the formula $\partial(f)(x)=\partial(f(x))-(-1)^{\text{deg}f}f(\partial(x))$, where $x\in A$, thus making $(A,B)$ into a complex.

Let $M$ and $N$ be finitely generated $R$-modules and $P_M$ and $P_N$ denote their projective resolutions, respectively. We shall use $P_M$ (resp. $P_N$) to denote the corresponding
underlying graded module. The subset $(P_M,P_N)_b$ of bounded homogeneous maps (a homogeneous map is called bounded if only finitely many components of that map are non-zero) is a graded submodule of $(P_M,P_N)$. The restriction of $\partial$ to $(P_M,P_N)_b$ make it into a subcomplex of $(P_M,P_N)$. The corresponding quotient
complex will be of fundamental important to us. We denote by $(\widetilde{P_M,P_N})$ the quotient complex
$$(\widetilde{P_M,P_N})=(P_M, P_N)/(P_M, P_N)_b.$$
Passing on to cohomology we obtain Vogel cohomology. It will be denoted by $\E_R^*(M,N)$. Moreover the short exact sequence
$$0\rightarrow (P_M, P_N)_b\rightarrow (\widetilde{P_M,P_N})\rightarrow 0,$$
where the cohomology of the middle term is just $\e_R^*(M, N)$, yields upon passing to corresponding long cohomology exact sequence, a natural transformation $\z^*(M, N): \e_R^*(M, N)\rightarrow \E_R^*(M, N).$

By essentially following the same argument analogous to ordinary cohomology, one
can see that $\E_R^*$ is a cohomological functor, independent of the choice of projective resolutions of $M$ and $N [G, I]$. The following result that will be used latter, is easy to see.
\begin{proposition}
Let $(R,\mathfrak{m})$ be a commutative Noetherian local ring. Then for any $R$-module $M$ the following are equivalent:
\begin{itemize}
\item[i)] $\p_RM$ is finite.
\item[ii)] $\E_R^i(M,~)=0$ for all integer $i$.
\item[iii)] $\E_R^i(~,M)=0$ for all integer $i$.
\end{itemize}
\end{proposition} 
For simplicity, for any $R$-module $M$ and any integer $i$ we let $\z^i(M)$ denote the natural transformation $\z^i(M, k):\e_R^n(M, k)\rightarrow \E_R^n(M,k).$
\begin{definition}
Let $M \neq 0$ be a finitely generated $R$-module. We assign an invariant to $M$, called $\z$-dimension of $M$, denoted $\z\text{-}\dm_RM$, by the formula
$$\z\text{-}\dm_RM= \text{inf}\Bigg\{~n\begin{array}{c|c}  & \z^n(M)~\text{is  epimorphism and}~ \z^i(M)~\text{is isomorphism}\\ &~\text{for all}~i>n\end{array}\Bigg\}.$$
We complement this by setting $\z\text{-}\dm_R0=-\infty$.
\end{definition}
Next theorem is our first main result.
\begin{theorem}
Let $(R,\mathfrak{m}, k)$ be a commutative Noetherian local ring. Let $n\in \mathbb{N}$. Then the following conditions are equivalent:
\begin{itemize}
\item[i)] $\z\text{-}\dm_RM\leq n$, for all finitely generated $R$-module $M$.
\item[ii)] $\z\text{-}\dm_Rk\leq n$.
\item[iii)] $\z^i(k)$ is epimorphism, for all $i\geq n$.
\end{itemize}
\end{theorem}
\begin{proof}
The implications $(i) \Rightarrow (ii)$ and $(i) \Rightarrow (iii)$ are trivially hold.

$(ii)\Rightarrow (i)$. Let $M$ be a finitely generated $R$-module. We induce on $\dm M$. Suppose first $\dm M = 0$. So $l(M)$, the lengths of $M$ is finite, say $s$. We use induction on s to prove the result in this case. If $s = 1$, there is nothing to prove.
So let $s > 1$, and consider the short exact sequence
$$0\rightarrow k\rightarrow M\rightarrow N\rightarrow 0,$$
where $l(N) = s - 1$. For any integer $i$, there exists a commutative diagram of Rmodules and $R$-homomorphisms
$$
\begin{tikzcd}
\cdots\arrow{r} & \e_R^{i-1}(k, k)\arrow{r}\arrow{d}{\z^{i-1}(k)} & \e_R^i(N, k)\arrow{r}\arrow{d}{\z^{i}(N)} &  \e_R^i(M, k)\arrow{r}\arrow{d}{\z^{i}(M)}& \e_R^i(k, k)\arrow{r}\arrow{d}{\z^{i}(k)} &\cdots\\
\cdots\arrow{r} & \E_R^{i-1}(k, k)\arrow{r}& \E_R^i(N, k)\arrow{r}&  \E_R^i(M, k)\arrow{r}& \E_R^i(k, k)\arrow{r} &\cdots
\end{tikzcd}
$$
By induction assumption, $\z^i(k)$ and $\z^i(N)$ are both epimorphism for $i = n$ and isomorphism for $i > n$. This by a simple diagram chasing, in view of the Five Lemma, will implies that $\z^i(M)$ is epimorphism for $i = n$ and isomorphism for $i > n$. This
completes the proof in this case. Now suppose, inductively, that $\dm M = n > 0$ and the result has been proved for all $R$-modules of dimension less than $n$. Consider the short exact sequence
$$0\rightarrow \Gamma_m(M)\rightarrow M\rightarrow M/\Gamma_m(M)\rightarrow 0$$
of $R$-modules, where $\Gamma_m(M)$denotes the $\mathfrak{m}$-torsion functor $\cup_{n\in\mathbb{N}}0:_{(M)} \mathfrak{m}^n$. This in turn induces, for any integer $i$ a commutative diagram of $R$-modules and $R$-homomorphisms
$$
\begin{tikzcd}
 \e_R^{i-1}(\Gamma_m(M), k)\arrow{r}\arrow{d}{\z^{i-1}(\Gamma_m(M))} & \e_R^i(\frac{M}{\Gamma_m(M)}, k)\arrow{r}\arrow{d}{\z^{i}(M/\Gamma_m(M))} &  \e_R^i(M, k)\arrow{r}\arrow{d}{\z^{i}(M)}& \e_R^i(\Gamma_m(M), k)\arrow{d}{\z^{i}(\Gamma_m(M))} \\
\E_R^{i-1}(\Gamma_m(M), k)\arrow{r}& \E_R^i(\frac{M}{\Gamma_m(M)}, k)\arrow{r}&  \E_R^i(M, k)\arrow{r}& \E_R^i(\Gamma_m(M), k) 
\end{tikzcd}
$$
By inductive assumption, $\z^n(\Gamma_m(M))$ is epimorphism and $\z^i(\Gamma_m(M)$ is isomorphism for all $i > n$. So $\z\text{-}\dm_RM \leq n$ if and only if $\z\text{-}\dm_R(M/\Gamma_m(M)) \leq n$. We can therefore assume, in the inductive step that there exists an element $r \in R$ which
is a non-zero divisor on $M$. The exact sequence
$$0\rightarrow M\xrightarrow{r} M\rightarrow M/rM\rightarrow 0$$
induces for any integer $i$, a commutative diagram of $R$-modules and $R$-homomorphisms
$$
\begin{tikzcd}
 \e_R^{i}(M/rM, k)\arrow{r}\arrow{d}{\z^{i}(M/rM)} & \e_R^i(M, k)\arrow{r}\arrow{d}{\z^{i}(M)} &  \e_R^i(M, k)\arrow{r}\arrow{d}{\z^{i}(M)}& \e_R^{i+1}(M/rM, k)\arrow{d}{\z^{i+1}(M/rM)} \\
\E_R^{i}(M/rM, k)\arrow{r}& \E_R^i(M, k)\arrow{r}&  \E_R^i(M, k)\arrow{r}& \E_R^{i+1}(M/rM, k) 
\end{tikzcd}
$$
Since $\dm M/rM < \dm M$, by inductive assumption, $\z\text{-}\dm_R(M/rM) \leq  n$. So
$\z^n(M/rM)$ is epimorphism and $\z^i(M/rM)$ is isomorphism for all $i > n$. But $r \in \mathfrak{m}$
and so each element of $\E_R^i(M, k)$ is annihilated by multiplication by $r$. This fact in conjunction with the latter diagram, will implies that $\z^n(M)$ is epimorphism and $\z^i(M)$ is isomorphism for all $i > n$.

$(iii)\Rightarrow (ii)$. If $R$ is regular, the result is clear, because $\p_Rk$ is finite and so $\E_R^i(k, k)=0$ for all integer $i$. So let $R$ is non-regular. Then by [M2, Theorem 6],
$\z^i(k)$ is monomorphism for all integer $i$. This follows the result.
\end{proof} 

Following result shows that $\z\text{-}\dm_RM$ is a refinement of $G$-dimension. We preface
it by recalling the structure of Tate cohomology, introduced through complete resolutions, that has been the subject of several recent expositions, in particular by
Buchweitz [B] and Cornick and Kropholler [CK]. LetM be a finite $R$-module of finite
$G$-dimension. Choose a complete resolution $T\xrightarrow{v} P\xrightarrow{\pi} M$ of M (see for instance [AM, Sec.5]). Then for each $R$-module $N$ and for each $n \in \Z$, Tate cohomology group is defined by the equality $$\widehat{\e}_R^n(M, N)=\text{H}^n \text{Hom}_R(T, N).$$

\begin{theorem}
For any finitely generated $R$-module $M$, there is an inequality $$\z\text{-}\dm_RM \leq \text{G}\text{-}\dm_RM$$
with equality, when G-$\dm_RM$ is finite.

\end{theorem}
\begin{proof}
Without loss of generality, we may assume that G-$\dm_RM = g$ is finite. With this assumption, by [M1, §2], for any integer $i$, there is a natural isomorphism of
cohomology functors $\widehat{\e}_R^i(M, k)\cong \E_R^i(M, k)$, compatible with the maps coming
from $\e_R^i(M, k)$. If $g = 0$, using the definition of Tate cohomology it is easily seen
that $\widehat{\e}_R^0(M, k)\cong \text{Hom}_R(M, k)/N$, for suitable $R$-module $N$. So $\z^0(M)$ is always
epimorphism. Moreover it follows from [AM, 5.2(2)] that $\z^i(M)$ is isomorphism
for all $i > 0$. Hence $\z\text{-}\dm_RM = 0$. Now let $g > 0$. It follows from [AM, 5.2(2)]
that $\z^i(M)$ is isomorphism for all $i > g$ and follows from [AM, 7.1] that $\z^g(M)$
is epimorphism. So $\z\text{-}\dm_RM \leq \text{G-}\dm_RM$. For equality, consider a Gorenstein
resolution of $M$, say
$$0\rightarrow P_{g}\rightarrow P_{g-1}\rightarrow\cdots\rightarrow P_{1}\rightarrow G_{0}\rightarrow M\rightarrow 0$$
with all the $P_i$'s finitely generated and projective and with $G_0$ of Gorenstein dimension
zero. But then we can also assume that each $P_i \rightarrow P_{i-1}$ gives a projective
cover of the image of $P_i$ in $P_{i-1}$. So in particular this means that $P_g \subset \mathfrak{m}P_{g-1}$. But
this implies that $\text{Hom}_R(P_{g-1}, k) \rightarrow \text{Hom}_R(P_g, k)$ is the zero map. So if $P_g \neq 0$ we see that $\e_\mathcal{G}^g(M, k)\neq 0$. Here we are tacitly assuming $g \geq 2$. If $g = 1$, then the
resolution of $M$ looks like $0\rightarrow F\rightarrow G\rightarrow M\rightarrow 0$. Here $F$ is a finitely generated free $R$-module. We can assume that $F \subset \mathfrak{m}G$, for if not one can use Nakayama's
lemma to get a copy of $R$ in $F$ a direct summand of $G$. But then we can go modulo this copy of $R$. So repeating if necessary, we finally get that $F \subset \mathfrak{m}G$. But then if $F \neq 0$, by the same type argument as above we get that $\e_{\mathcal{G}}^1(M, k)\neq 0$. Hence
by [AM, 7.1], either $\z^g(M)$ is not injective or $\z\text{-}\dm ^{g-1}(M)$ is not epimorphism. So
$\z\text{-}\dm_RM \geq g$. This completes the proof.

\end{proof}

Now we are in position to put all our results together to present a characterization
for Gorenstein rings in terms of $\z$-dimension. We need the following proposition.
\begin{proposition}
Let $M$ be a finitely generated $R$-module of finite $\z$-dimension,
say $n$. Then for any finitely generated $R$-module $N$, $\z^n(M,N)$ is epimorphism and
$\z^i(M,N)$ is isomorphism for all $i > n$.

\end{proposition}

\begin{proof}
Let $N$ be a finitely generated $R$-module. By following the same type argument as we have used for the proof of Theorem 1.3, we may assume inductively that $\dm N > 0$, the result holds for all finitely generated modules of dimension less
than $\dm N$ and also there exists a non-zerodivisor $r$ on $N$. So we have a short exact
sequence $0\rightarrow N\xrightarrow{r}N\rightarrow N/rN\rightarrow 0$. This induces, for any integer $i$, a commutative
diagram of $R$-modules and $R$-homomorphisms

$$
\begin{tikzcd}
 \e_R^{i}(M, N/rN)\arrow{r}\arrow{d}{\z^{i}(M,N/rN)} & \e_R^{i+1}(M, N)\arrow{r}\arrow{d}{\z^{i+1}(M, N)} &  \e_R^{i+1}(M, N)\arrow{r}\arrow{d}{\z^{i+1}(M, N)}& \e_R^{i+1}(M, N/rN)\arrow{d}{\z^{i+1}(M, N/rN)} \\
\E_R^{i}(M, N/rN)\arrow{r}& \E_R^{i+1}(M, N)\arrow{r}&  \E_R^{i+1}(M, N)\arrow{r}& \E_R^{i+1}(M, N/rN) 
\end{tikzcd}
$$

Since for any $i \geq n$, $\z^i(M,N/rN)$ is epimorphism and $\z^{i+1}(M,N/rN)$ is isomorphism,
by a diagram chasing one can see that the multiplication map by $r$ restricted to $\text{Ker}  \z^{i+1}(M,N)$ is epimorphism. So using Nakayama's Lemma, for any $i \geq n$,
we get that the map $\z^{i+1}(M,N)$ is monomorphism. Now consider the commutative
diagram

$$
\begin{tikzcd}
 \e_R^{i}(M, N)\arrow{r}{r}\arrow{d}{\z^{i}(M,N)} & \e_R^{i}(M, N)\arrow{r}\arrow{d}{\z^{i}(M, N)} &  \e_R^{i}(M, N/rN)\arrow{r}\arrow{d}{\z^{i}(M, N/rN)}& \e_R^{i+1}(M, N)\arrow{d}{\z^{i+1}(M, N)} \\
\E_R^{i}(M, N)\arrow{r}{r}& \E_R^{i}(M, N)\arrow{r}&  \E_R^{i}(M, N/rN)\arrow{r}& \E_R^{i+1}(M, N) 
\end{tikzcd}
$$
Since for any integer $i \geq n$, $\z^i(M,N/rN)$ is epimorphism and $\z^{i+1}(M,N)$ is
monomorphism, the restriction of the multiplication map $r$ to $\text{Coker}~ \z^i(M,N)$ is
surjective, and so by Nakayama's Lemma, $\text{Coker}~ \z^i(M,N) = 0$. Therefore $\z^i(M,N)$
is epimorphism for all $i \geq n$. This completes the proof.
\end{proof}
\begin{theorem}
The following conditions are equivalent:
\begin{itemize}
\item[i)] $R$ is Gorenstein.
\item[ii)] $\z\text{-}\dm_RM < \infty$, for all finitely generated $R$-module $M$.
\item[iii)]  $\z\text{-}\dm_R k < \infty$.
\end{itemize}

\end{theorem}

\begin{proof}
 $(i) \Rightarrow (ii)$. Since $R$ is Gorenstein, by [AB], G-$\dm_RM < \infty$, for all finitely
generated $R$-module $M$. So by Theorem 1.4, $\z\text{-}\dm_RM < \infty$.

$(ii) \Rightarrow (iii)$. This trivially holds.

$(iii) \Rightarrow (i)$. Let $\z\text{-}\dm_R k < \infty$. By Proposition 1.5, $\z^i(k,R)$ is isomorphism for all
integer $i > \z\text{-}\dm_R k$. But $\E_R^i(k, R)=0$ for all $i$. This implies that $\e_R^i(k, R)=0$
for all $i$ large enough. So $R$ is Gorenstein.

\end{proof}

\begin{proposition}
Let $\mathfrak{p}$ be a prime ideal in $\textnormal{Spec}(R)$. Then for any finitely generated $R$-module $M$,
$$\z\text{-}\dm_{R_{\mathfrak{p}}}M_{\mathfrak{p}} \leq \z\text{-}\dm_RM.$$

\end{proposition}
\begin{proof}
For proof just one should note that both functors $\e_R^*$ and $\E_R^*$ are well
behaved under localization. 
\end{proof}

Now we aim to give a lower bound for $\z\text{-}\dm_RM$. There is a refinement of projective
dimension of $M$ denoted by $\ap_RM$, defined by the formula

$$\ap_RM=\text{sup} \Bigg\{ i\in \mathbb{N}_0 \begin{array}{c|c}  & \e_R^i(M,T)\neq 0~\text{for some finitely generated}\\ &~R\text{-}\text{module}~T~\text{with}~\p_RT<\infty\end{array}\Bigg\}.$$
It is proved in [AB, Theorem 4.13] that $\ap_R$ is also a refinement of G-dimension
G-$\dm_R$ . We shall show that $\ap_RM \leq \z\text{-}\dm_RM$, with equality where $\z\text{-}\dm_RM$
is finite.
\begin{theorem}
Let $M$ be a finitely generated $R$-module. Then
$$\ap_RM \leq \z\text{-}\dm_RM$$
with equality if $\z\text{-}\dm_RM$ is finite.

\end{theorem}

\begin{proof}
We may (and do) assume that $\z\text{-}\dm_RM = n$ is finite. So by Proposition 1.5,
for any finitely generated $R$-module $N$, $\z^n(M,N)$ is epimorphism and $\z^i(M,N)$ is
isomorphism for all $i > n$. Now let $N$ be an $R$-module of finite projective dimension.
By Proposition 1.1, $\E_R^i(M, N)=0$ for all integer $i$. So we get $\e_R^i(M, N)=0$
for all $i >n$. This implies that $\ap_RM \leq \z\text{-}\dm_RM$. Now let $\ap_R(M) = s$ and
$s < n$. We seek for a contradiction. The short exact sequence $0\rightarrow \mathfrak{m}\rightarrow R\rightarrow k\rightarrow 0$
induces the following commutative diagram

$$
\begin{tikzcd}
 \e_R^{n-1}(M, k)\arrow{r}\arrow{d}{\z^{n-1}(M,k)} & \e_R^{n}(M, \mathfrak{m})\arrow{r}\arrow{d}{\z^{n}(M, \mathfrak{m})} &  \e_R^{n}(M, k)\arrow{r}\arrow{d}{\z^{n}(M, R)}& \e_R^{n}(M, k)\arrow{r}\arrow{d}{\z^{n}(M, k)}& \e_R^{n+1}(M, \mathfrak{m})\arrow{d}{\z^{n+1}(M, \mathfrak{m})} \\
\E_R^{n-1}(M, k)\arrow{r}& \E_R^{n}(M, \mathfrak{m})\arrow{r}&  \E_R^{n}(M, R)\arrow{r}& \E_R^{n}(M, k)\arrow{r}& \E_R^{n+1}(M, \mathfrak{m})
\end{tikzcd}
$$
Since $s < n$, $\e_R^i(M, R)=0$ for all $i \geq n$. So since $\z^{n+1}(M,\mathfrak{m})$ is isomorphism, we
get $\z^n(M, k)$ is isomorphism and since $\z^n(M,\mathfrak{m})$ is epimorphism, we get $\z^{n-1}(M, k)$
is epimorphism. Therefore $\z\text{-}\dm_RM \leq n - 1$. This is the desired contradiction. So $s = n$.

\end{proof}

\begin{corollary}
For any finitely generated $R$-module $M$,
$$\ap_R(M)\leq \z\text{-}\dm_RM\leq G\text{-}\dm_RM\leq CI\text{-}\dm_RM\leq \p_RM,$$
with equality to the left of any finite ones.

\end{corollary}

Finally we show that the category of modules of finite $\z$-dimension has resolving
property. Let $\mathcal{Z}$ (resp. $\mathcal{\tilde{Z}}$) denotes the full subcategory of $\mathcal{F}$, the category of
finitely generated $R$-modules and $R$-homomorphisms, whose objects are modules of
$\z$-dimension zero (resp. of finite $\z$-dimension).

\begin{proposition}
The category $\mathcal{Z}$ is closed under extension and kernels of epimorphisms.
Moreover $\mathcal{Z}$ contains $\mathcal{P}$, the category of finite projective $R$-modules.

\end{proposition}
\begin{proof}
Let $0\rightarrow E\rightarrow L\rightarrow M\rightarrow 0$ be an exact sequence of $R$-modules with
$M \in \mathcal{Z}$. We shall show that $L \in \mathcal{Z}$ if and only if $E \in \mathcal{Z}$. The above short exact
sequence induces for any integer $i$, a commutative diagram of $R$-modules and $R$-homomorphisms
$$
\begin{tikzcd}
 \e_R^{i}(M, k)\arrow{r}\arrow{d}{\z^{i}(M)} & \e_R^{i}(L, k)\arrow{r}\arrow{d}{\z^{i}(L)} &  \e_R^{i}(E, k)\arrow{r}\arrow{d}{\z^{i}(E)}& \e_R^{i+1}(M, k)\arrow{d}{\z^{i+1}(M)} \\
\E_R^{i}(M, k)\arrow{r}& \E_R^{i}(L, k)\arrow{r}&  \E_R^{i}(E, k)\arrow{r}& \E_R^{i+1}(M, k) 
\end{tikzcd}
$$
It is now easy to deduce the first assertion, by a simple diagram chasing and also
five lemma. The last assertion is elementary.
\end{proof}
\begin{proposition}

The category $\mathcal{\tilde{Z}}$ is closed under extension, kernel of epimorphisms
and cokernel of monomorphisms. Moreover $\mathcal{\tilde{Z}} \supseteq \mathcal{\tilde{G}}$, where $\mathcal{\tilde{G}}$ denotes the subcategory
of $\mathcal{F}$ whose objects are finitely generated $R$-modules of finite $G$-dimension.

\end{proposition}
\begin{proof}
Last assertion follows from Theorem 1.4. The other ones are easy to see.
\end{proof}

\end{document}